\documentclass[runningheads]{llncs}

\usepackage[T1]{fontenc}
\usepackage{amsmath,amssymb,mathtools}
\usepackage{microtype}
\usepackage{enumitem}
\usepackage{hyperref}
\usepackage{xcolor}
\newcounter{claim}

\hypersetup{
    colorlinks=true,
    linkcolor=blue,
    urlcolor=blue,
    citecolor=blue
}

\title{Strong Edge Colouring of Disk Graphs: A \texorpdfstring{$6$}{6}-Approximation and an Improved Unit-Disk Bound}
\titlerunning{Strong Edge Colouring of Disk Graphs}

\author{Sandip Das \and
Sk Samim Islam \and
Aashirwad Mohapatra\and Sangita Saha\and
Saumya Sen}

\authorrunning{S. Das et al.}

\institute{
Indian Statistical Institute, Kolkata, India\\
\email{sandip.das.69@gmail.com}\
\email{samimislam08@gmail.com}\
\email{aashirwad\_r@isical.ac.in}\
\email{sangitasaha543@gmail.com}\
\email{saumyasen72@gmail.com}
}

\begin{document}
\maketitle

\begin{abstract}
A \emph{strong edge colouring} of a graph $G$ is an edge colouring in which every colour class is an induced matching. The minimum number
of colours is the \emph{strong chromatic index} $\chi'_s(G)$. If each edge
$e$ is assigned a list $L'(e)$ and its colour must belong to $L'(e)$, the corresponding parameter is the \emph{strong list chromatic index}
$\chi'_{s,\ell}(G)$. From the definitions,  $\chi'_s(G)\le\chi'_{s,\ell}(G)$. 

Barrett 
et al. gave an $8$-approximation for strong edge colouring on unit disk graphs and Grelier et al. improved the approximation factor to $6$. Our first result extends this factor-$6$ guarantee from unit disk graphs to the strictly larger class of disk graphs.

In another direction, Erd\H{o}s and Ne\v{s}et\v{r}il conjectured that the strong chromatic index of a graph of maximum degree $\Delta$ is asymptotically at most $1.25\Delta^2$. The best published general asymptotic upper bound has leading coefficient $1.772$, due to Hurley et al. For unit disk graphs, D\k{e}bski et al. proved that $\chi'_s(G) \leq 1.625 \Delta^2$. Our second result improves this leading coefficient to $225/142 \approx 1.5845$. In fact, the proof establishes a stronger bound $\chi'_{s,\ell}(G)\le\frac{225}{142} \Delta^2+O(\Delta)$ for unit disk graphs. 

\end{abstract}

\section{Introduction}
A \emph{strong edge colouring} of a graph $G$ is an edge colouring in which
every colour class is an induced matching. Equivalently, it is a proper vertex colouring of $L(G)^2$, the square of the line graph. The minimum number
of colours is the \emph{strong chromatic index}, $\chi'_s(G)$. For each edge $e \in E(G)$, let $L'(e)$ be the list of available colours of $e$, and let $L'=\{L'(e) : e\in E(G)\}$. The graph $G$ is \textit{strongly $L'$-edge-colourable} if there exists a strong edge colouring $c$ of $G$ such that $c(e) \in L'(e)$ for each $e \in E(G)$. For a positive integer $k$, a graph $G$ is \textit{strongly $k$-edge-choosable} if $G$ is  strongly $L'$-edge-colourable for every $L'$ with  $|L'(e)|\geq k$ for all $e \in E(G)$.
The \emph{strong list chromatic index}, $\chi'_{s,\ell}(G)$ is the minimum positive integer $k$ for which $G$ is  strongly $k$-edge-choosable. Note that $\chi'_s(G)\leq\chi'_{s,\ell}(G)$ for every graph $G$.

A \emph{disk graph} is the intersection graph of finitely many closed
Euclidean disks of arbitrary positive radii. A \emph{unit disk graph} is the
special case in which all radii are equal. 
Barrett et al.~\cite{BarrettEtAl2006} obtained an $8$-approximation for
Strong Edge Colouring on unit disk graphs. Later, Grelier et al.~\cite{grelier2019approximate} improved the approximation factor to $6$. They also proved that Strong Edge $6$-Colourability is
NP-complete on unit disk graphs. Consequently, there is no polynomial-time
$(7/6-\varepsilon)$-approximation unless $P=NP$.  We show that the same approximation factor remains valid when the equal-radius
assumption is removed completely. Thus our first result extends the best known unit-disk approximation guarantee to disk graphs.

\begin{theorem}\label{thm:disk}
There is a polynomial-time $6$-approximation algorithm for Strong Edge Colouring on disk graphs. Moreover, the algorithm does not require a disk representation of $G$.
\end{theorem}

In 1985 Erd\H{o}s--Ne\v{s}et\v{r}il conjectured the following. If $G$ is a graph with maximum degree $\Delta$, then  $\chi'_s(G)\le \frac{5 \Delta^2}{4}$ for even $\Delta$ and for odd $\Delta$, $\chi'_s(G)\le\frac{5 \Delta^2}{4} - \frac{2 \Delta - 1}{4}$ 
\cite{ErdosNesetril}. The best published general asymptotic upper bound is $1.772\Delta^2$ for sufficiently large $\Delta$, due to Hurley, de Joannis de Verclos, and Kang~\cite{HurleyEtAl2021,HurleyEtAl2022}. A recent preprint of Davey et al.~\cite{davey2026strong} further improves the bound to $1.73\Delta^2$.
For unit disk graphs, D\k{e}bski, Junosza-Szaniawski, and \'Sleszy\'nska-Nowak~\cite{DebskiEtAl2020} proved that
$\chi'_s(G)\leq 13\Delta^2/8=1.625\Delta^2$. They also constructed a family
of unit disk graphs having the strong chromatic index at least $\frac{9}{16}\Delta^2$. Thus,  the leading coefficient for the strong chromatic index of unit disk graphs is known to lie
between $\frac{9}{16}$ and $\frac{13}{8}$. In the same paper, they asked whether a
strong edge colouring of a unit disk graph with close to $1.625\Delta^2$ colours can be obtained by a polynomial-time algorithm that does not use a disk representation. For sufficiently large $\Delta$, we answer their question affirmatively and also improve the leading coefficient from $1.625$ to $1.585$. In fact, we prove the stronger result that $\chi'_{s,\ell}(G)\le \frac{225}{142}\Delta^2+O(\Delta)$.

\begin{theorem} \label{thm:unit disk}
 Each unit disk graph $G$ of maximum degree $\Delta$ satisfies  $\chi'_s(G)\le
  \frac{225}{142} \Delta^2+ O(\Delta)$.
\end{theorem}

For Theorem~\ref{thm:disk}, we prove the degeneracy
bound $\delta^*(L(G)^2)\le6(\chi'_s(G)-1)-1$ and obtain an algorithm.

\medskip

\noindent\textbf{Organization: } In Section~\ref{sec:preliminaries}, we introduce the notation and preliminary definitions used throughout the paper. In Section~\ref{sec:proofthm1}, we prove Theorem~\ref{thm:disk} and establish the polynomial-time $6$-approximation algorithm for Strong Edge Colouring on disk graphs. In Section~\ref{sec:proofthm2}, we prove Theorem~\ref{thm:unit disk} and derive the improved upper bound on the strong list chromatic index of unit disk graphs. Finally, we conclude in Section~\ref{sec:conclusion}.

\section{Preliminaries}\label{sec:preliminaries}

Let $G=(V(G), E(G))$ be a graph with vertex set $V(G)$ and edge set $E(G)$. All graphs are finite and simple. Let $F\subseteq V(G)$, then the \emph{induced subgraph} $G[F]$ is the graph whose vertex set is $F$ and whose edge set consists of all of the edges in $E(G)$ that have both endpoints in $F$. For a graph $G$, its \emph{degeneracy} is $\delta^*(G)=\max\{\delta(H):H\text{ is an induced subgraph of }G\}$, where $\delta(H)$ is minimum degree of a graph $H$. An \emph{independent set} of a graph $G$ is a set of vertices such that no two of which are adjacent. A maximum independent set is an independent set of largest possible size for a given graph $G$. This size is called the \emph{independence number} of $G$ and denoted by $\alpha(G)$. A graph $G=(V, E)$ is \emph{bipartite} if the vertex set $V$ is partitioned into $A$ and $B$  such that there is no edge between any two vertices of $A$ or any two vertices of $B$ and we denote this graph as $G=(A,B,E)$. A set $M\subseteq E(G)$ is a \emph{matching} in $G$ if no two distinct edges of $M$ have a common endpoint. A matching $M$ is an \emph{induced matching} if, in addition, no edge of $G$ joins an endpoint of one edge of $M$ to an endpoint of another edge of $M$.

The \emph{line graph} $L(G)$ of $G$ is the graph where a vertex in $L(G)$ corresponds to an edge in $G$ and there is an edge between two vertices in $L(G)$ if the corresponding edges in $G$ share a vertex.
The \emph{square} $G^2$ of a graph $G$ is the graph with vertex set $V(G)$ in which two distinct vertices are adjacent if and only if their distance in $G$ is at most $2$. Two distinct edges of $G$ \emph{conflict} if they are adjacent in $L(G)^2$, i.e., when two edges of $G$ share an endpoint or two edges are joined by another edge of $G$.

For a disk $D$, let $\partial D$ and $\operatorname{int}(D)$ denote its boundary and interior respectively.
We say that a set of disks of a disk representation $\mathcal D=\{D_v:v\in V(G)\}$ of a disk graph $G$ are in \textit{general-position} if: (i) For every pair of distinct vertices $u,v\in V(G)$, if $uv\in E(G)$, then $\operatorname{int}(D_u)\cap\operatorname{int}(D_v)\neq\emptyset$. Moreover, whenever $\partial D_u\cap\partial D_v\neq\emptyset$, the two boundaries cross at every point of intersection. 
(ii) No three disk boundaries contain a common point.
(iii) No two disks coincide.
(iv)Every disk can be expanded by a sufficiently small positive amount without changing the arrangement of the disks combinatorially.
Here, an expansion of a disk with centre $c$ and radius $r$ means replacing it by the concentric disk with radius $r+\varepsilon$ for some $\varepsilon>0$.
\begin{lemma}\label{lem:gdp}
Every finite disk graph has a disk representation satisfying general-position conditions.
\end{lemma}

\begin{proof}
Let $\mathcal D=\{D_v:v\in V(G)\}$ be an arbitrary disk representation of $G$, where $D_v$ has centre $c_v$ and radius $r_v$. For every $uv\notin E(G)$, $ \|c_u-c_v\|-r_u-r_v>0.$ Since $G$ is finite, the minimum of these quantities over all pairs of distinct nonadjacent vertices is positive. Choose $\delta>0$ smaller than half this minimum and replace each $r_v$ by $r_v+\delta$. If $G$ is complete, choose any $\delta>0$. This keeps the disks corresponding to nonadjacent vertices disjoint, whereas the disks corresponding to adjacent vertices intersect in their interiors. At this point, disks corresponding to adjacent vertices intersect in their interiors, whereas disks corresponding to nonadjacent vertices are separated by a positive distance. Adjust the radii one disk at a time, keeping the disks already considered fixed. For the disk currently under consideration, there are only finitely
many forbidden values of its radius: those for which its boundary is tangent to a previously fixed boundary, the disk coincides with a previously fixed disk, or its boundary passes through an intersection point of two previously fixed boundaries. Choose a radius avoiding these values. Since all changes are sufficiently small, disks corresponding to adjacent vertices
continue to intersect in their interiors, and disks corresponding to nonadjacent vertices remain disjoint. After all radii have been adjusted, no
two boundaries are tangent, no three boundaries contain a common point, and no two disks coincide. Hence the represented graph is unchanged and
conditions~{\rm(i)--(iii)} hold.

It remains to verify condition~{\rm(iv)}. Fix a disk $D\in\mathcal D$ with
centre $c$ and radius $r$. For $\varepsilon\geq0$, let $D^\varepsilon$ denote the concentric disk with radius $r+\varepsilon$, while all other disks
remain unchanged. As $\varepsilon$ varies, the combinatorial arrangement can change only if $\partial D^\varepsilon$ becomes tangent to or coincides with
another boundary, or contains an intersection point of two other boundaries. Since $\mathcal D$ is finite and two circles have at most two common points,
only finitely many values of $\varepsilon>0$ produce one of these events. Conditions~{\rm(i)--(iii)} ensure that none of these events occurs for
$\varepsilon=0$. Hence there exists $\varepsilon_D>0$, smaller than every positive value at which such an event occurs, such that replacing $D$ by
$D^\varepsilon$ for any $0<\varepsilon<\varepsilon_D$ does not change the arrangement of the disks combinatorially. Therefore,
condition~{\rm(iv)} holds for every $D\in\mathcal D$.
\hfill$\square$
\end{proof}

A set of points $\mathcal{P}$ and a family of disks $\mathcal{D}$ in the plane define a  hypergraph, $\mathcal {H}(\mathcal{P}, \mathcal{D})$, where each disk $D \in \mathcal{D}$ defines a hyperedge $\mathcal{P}\cap D$. Given a family $\mathcal{R}$ of red disks and $\mathcal{B}$ of blue disks, we define \emph{intersection hypergraph} $\mathcal{H}(\mathcal{B}, \mathcal{R})$ as follows: the disks $\mathcal{B}$ are the vertices and each disk $W \in \mathcal{R}$ defines a hyperedge $\{ D \in \mathcal{B}: D \cap W \neq \emptyset\}$. A \emph{support} for $\mathcal{H}(\mathcal{B}, \mathcal{R})$ is a graph $S$
on vertex set $\mathcal B$ in which every hyperedge induces a connected subgraph of $S$. $S$ is called a \emph{planar support} if $S$ is planar. Two regions $X, Y \subseteq \mathbb{R}^2$ are said to be \emph{non-piercing} if both $X\setminus Y$ and $Y\setminus X$ are connected. A family $\Gamma$ of regions is said to be \emph{non-piercing} if the regions in $\Gamma$ are pairwise non-piercing. Every family of closed Euclidean disks is non-piercing.

\section{Proof of Theorem 1}\label{sec:proofthm1}

 In this section, we prove Theorem~\ref{thm:disk}. We first establish a bound on the number of conflicts between two induced matchings in a disk graph $G$. We then apply this bound to $L(G)^2$ and obtain a polynomial-time $6$-approximation algorithm that does not require a disk representation of $G$.
 We use the following result of Raman and Ray~\cite{RamanRay2018} to prove Lemma~\ref{lem:ind match}.
\begin{lemma}[\cite{RamanRay2018}]\label{thm:planar-support}
Given two families $\mathcal B$ and $\mathcal R$ of non-piercing regions, the intersection hypergraph $\mathcal{H}(\mathcal{B}, \mathcal{R})$ admits a planar support. 
\end{lemma}
We now bound the number of conflicts between two induced matchings of a disk
graph.
\begin{lemma}\label{lem:ind match}
Let $A$ and $B$ be two nonempty disjoint induced matchings in a disk graph $G$, and let $J$ be the bipartite graph with partite sets $A$ and $B$ where $e \in A$ and $f \in B$ are adjacent if $e$ and $f$ conflict. Then $|E(J)|\leq3\bigl(|A|+|B|\bigr)-4$.
\end{lemma}

\begin{proof}
If $|A|+|B|=2$, then $|E(J)|\leq1$, and the result follows. Hence assume that $|A|+|B|\geq3$. For a disk graph $G=(V(G),E(G))$, fix a disk representation $\mathcal D=\{D_v:v\in V(G)\}$, where $uv\in E(G)$ exactly when $D_u\cap D_v\ne\emptyset$. For an edge $e=uv$, we define  $R_e=D_u\cup D_v$. $\partial R_e$ is defined as the boundary of  $R_e$. Suppose that $R_e\subseteq R_f$ for some $e\in A$ and $f\in B$. Since  $B$ is an induced matching we have that, $R_e\cap R_g\subseteq R_f\cap R_g=\emptyset$ for every
$g\in B\setminus\{f\}$. Thus $e$ has degree exactly one in $J$. 
For an edge $ef\in E(J)$, remove $e$ from $J$ if $R_e\subseteq R_f$. Similarly, we can remove $f$ from $J$ if $R_f\subseteq R_e$. By the preceding observation, this vertex has degree one in the current graph. Continue this process until every remaining edge $ef\in E(J)$ satisfies
$R_e\nsubseteq R_f$ and $R_f\nsubseteq R_e$. Every deletion removes exactly one vertex and one edge from $J$. If one partite set becomes empty, at most $|A|+|B|-1$ edges of $J$ have
been removed, and therefore
$|E(J)|\leq |A|+|B|-1\leq3(|A|+|B|)-4$ .
It remains to consider the case in which both induced matchings are nonempty and for every $ef\in E(J)$, neither $R_e\subseteq R_f$ nor $R_f\subseteq R_e$. We continue to denote the remaining induced matchings by $A$ and $B$. After deletion if $|A|+|B|=2$, then $|E(J)|\leq1\leq3(|A|+|B|)-4$ and restoring the deleted vertices as at the end of the proof preserves the inequality. Hence, for the remainder of the proof, we  assume that $|A|+|B|\geq3$.

Let $h$ be the number of edges $ef\in E(J)$ such that $e \in A$ and $f \in B$ have a common endpoint in $G$. These $h$ edges form a matching in $J$. Hence they have exactly $2h$ distinct endpoints in $A\cup B$.

We now construct an intersection hypergraph whose planar support will be used to bound $|E(J)|$. Consider the disks corresponding to the distinct endpoints of the edges in $A\cup B$. Since $A$ and $B$ are matchings and exactly $h$ vertices are common endpoints of an edge of $A$ and an edge of $B$, there are $2(|A|+|B|)-h$
such disks. Delete $h$ disks corresponding to these common endpoints, and let $\mathcal B$ be the family of the remaining disks. Hence $|\mathcal B|=2(|A|+|B|)-2h$. Every disk in $\mathcal B$ is associated with exactly one edge of $A\cup B$. Moreover, if $v$ is a common endpoint of $e\in A$ and $f\in B$, then $D_v$ is disjoint from both $D_a$ and $D_b$ for every edge $ab\in(A\setminus\{e\})\cup(B\setminus\{f\})$.

For every edge $e=uv\in E(G)$, the disks $D_u$ and $D_v$ intersect in their interiors, so the boundary of $R_e=D_u\cup D_v$ is a simple closed curve. Consider an edge $ef\in E(J)$, where $e=uv\in A$ and $f=xy\in B$, such that $e$ and $f$ have no common endpoint. Since $e$ and $f$ conflict, $R_e\cap R_f\neq\emptyset$. Moreover, neither $R_e\subseteq R_f$ nor
$R_f\subseteq R_e$. If $\partial R_e$ and $\partial R_f$ were disjoint, then the two regions would either be disjoint or one would be contained in the other. Therefore, $\partial R_e\cap\partial R_f\neq\emptyset$. Choose a point $p\in\partial R_e\cap\partial R_f$. The point $p$ lies on the boundary of exactly one of $D_u,D_v$ and exactly one of $D_x,D_y$ by Lemma~\ref{lem:gdp}. After interchanging the endpoints if necessary, assume that $p\in\partial D_u\cap\partial D_y$.
We next show that $D_u,D_y\in\mathcal B$. Suppose that the disk $D_u$ associated with $e$ was deleted. Then $u$ is a common endpoint of $e$ and some edge $g\in B \setminus \{f\}$. As $B$ is an induced matching, $D_u$ is disjoint from both $D_x$ and $D_y$, contradicting $p\in D_u\cap D_y$. Hence $D_u\in\mathcal B$ and also $D_y\in\mathcal B$. Furthermore, $p\notin D_v$ and $p\notin D_x$. Otherwise, $p$ would either lie in the interior of $R_e$ or $R_f$, contradicting $p\in\partial R_e\cap\partial R_f$, or belong to three disk
boundaries, contradicting Lemma~\ref{lem:gdp}. Moreover, $p$ belongs to no other disk of $\mathcal B$ as $A$ and $B$ are induced matchings. Since $\mathcal B$ is finite, there exists a sufficiently small closed disk
$W_{ef}$ centred at $p$ that intersects exactly $D_u$ and $D_y$ among the disks in $\mathcal B$. Consequently, the hyperedge determined by $W_{ef}$ is $\{D_u,D_y\}$.

\begin{figure}[ht]
    \centering
    \includegraphics[width=0.5\linewidth]{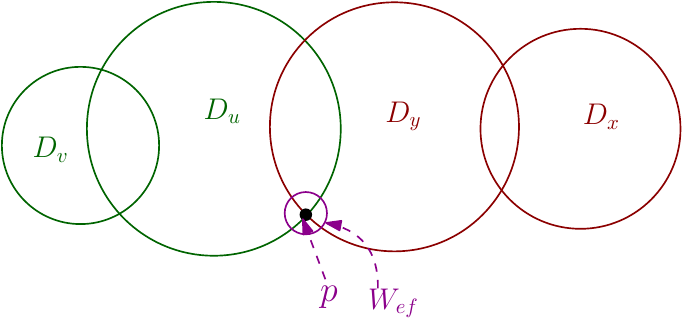}
    \caption{$e=uv \in A$, $f=xy \in B$, the disk $W_{ef}$ centred at $p \in \partial D_u \cap \partial D_y $.}
    \label{W_ef}
\end{figure}

Now let $e=uv\in A\cup B$ be an edge that is not incident with any of the $h$ edges of $J$ corresponding to a pair with a common endpoint. Then both $D_u$ and $D_v$ belong to $\mathcal{B}$. Choose a point $z\in\operatorname{int}(D_u\cap D_v)$ and take a sufficiently small closed disk $U_e$ centred at $z$ such that $U_e\subseteq\operatorname{int}(D_u\cap D_v)$.
Let $e \in A$, $f=xy, g=ab \in B$, then $R_f$ and $R_g$ are disjoint. Since the matching is finite, $U_e$
can be chosen sufficiently small so that $U_e \cap R_f \neq \emptyset$ for at most one such edge $f \in B$. Any disk $D_x$, $D_y$ intersected by $U_e$ belongs to $\mathcal{B}$. Consequently, the hyperedge defined by $U_e$ has one of the forms
$\{D_u,D_v\}, \{D_u,D_v,D_x\}, \{D_u,D_v,D_x,D_y\}$.
Whenever an edge $f$ occurs in this hyperedge, $ef\in E(J)$, because every point of $U_e\cap D_x$ or $U_e\cap D_y$ belongs to both $D_u$ and $D_v$. Furthermore, $e$ and $f$ have no common endpoint by the choice of $e$.
\begin{figure}[ht]
    \centering
\includegraphics[width=0.4\linewidth]{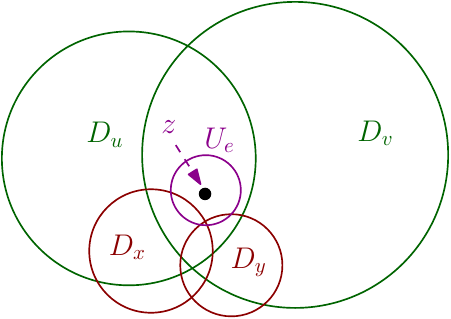}
    \caption{$e=uv \in A$, the disk $U_e$ is contained in $\operatorname{int}(D_u\cap D_v)$ and may intersect disks corresponding to at most one edge $f=xy \in B$. }
    \label{U_e}
\end{figure}

Let $\mathcal{R}$ consist of all disks $W_{ef}$ chosen for the edges $ef\in E(J)$ whose corresponding edges of $G$ have no common endpoint, together with all disks $U_e$ chosen for the elements of $A\cup B$ that are not
incident with any of the above $h$ edges of $J$. By Lemma~\ref{lem:gdp}, disks in $\mathcal B$ are in
general position. The choices of the  disks $W_{ef}$ and $U_e$ are not unique: their centres are fixed, but their radii may be chosen sufficiently small. We make these choices so that the properties of these disks given above are preserved and the resulting disks are in general position. Choose these radii one at a time. Once the preceding radii have been fixed, only finitely many values of the next radius create a
tangency, coincident boundaries, or a boundary passing through an intersection point of two previously fixed boundaries. Avoiding these values does not affect the properties of these disks established above. Consequently, the boundaries of any two disks in $\mathcal B\cup\mathcal R$ cross whenever they intersect, no three boundaries contain a common point, and no two disks coincide. Since $\mathcal B\cup\mathcal R$ is finite, every disk admits a sufficiently small positive expansion that does not change the arrangement
of these disks combinatorially. Thus, $\mathcal B\cup\mathcal R$ satisfies the general-position assumptions of ~\cite{RamanRay2018}. Both $\mathcal B$ and $\mathcal R$ are families of Euclidean disks and hence are non-piercing. Therefore, by Lemma~\ref{thm:planar-support}, their intersection hypergraph admits a planar support $S$ with vertex set $\mathcal B$.
We may take $S$ to be simple. Every disk $W_{ef}$ defines a hyperedge of size two ${D_u,D_y}$. Since this hyperedge must induce a connected subgraph of $S$, it follows that $D_u-D_y\in E(S)$.
Let $Q$ be the set of elements $e=uv\in A\cup B$ that are not incident with any of the above $h$ edges of $J$ and for which $D_u - D_v\in E(S)$.
For every edge $e=uv\in A\cup B$ that is not incident with any of the above $h$ edges of $J$ and does not belong to $Q$, choose a shortest path from $D_u$ to $D_v$ in the subgraph of $S$ induced by the disks in $\mathcal B$ that intersect $U_e$. By the possible forms of the hyperedge defined by $U_e$, this path has the form $D_u-D_x-D_v$ or $D_u-D_x-D_y-D_v$, where $f=xy$ belongs to the matching other than the one containing $e$. We say that $e$ selects $f$. In the path $D_u-D_x-D_y-D_v$ , $D_x-D_y\in E(S)$, and hence $f\in Q$. Since one path is selected for each relevant edge outside $Q$, the number of
selections is $|A| + |B| - 2 h - |Q|$.

Now construct a graph $T$ from $S$ such that $V(T)=V(S)$ and select $E(T)$ as follows:
\begin{enumerate}
\item For every edge $e=uv\in Q$, include $D_u-D_v$ in $E(T)$. \label{const:1}
\item  For every selected path corresponding to a selection between $e=uv$ and $f=xy$, include its two edges having one endpoint in $\{D_u,D_v\}$ and the other in $\{D_x,D_y\}$. \label{const:2}
\item For every edge $ef\in E(J)$, where $e=uv$ and $f=xy$ have no  common endpoint in $G$, include the edge of $S$ determined by $W_{ef}$ if  neither
$e$ selects $f$ nor $f$ selects $e$. \label{const:3}
\end{enumerate}
Since $T$ is a subgraph of $S$, it is planar.

\medskip
\refstepcounter{claim}
\noindent\textbf{Claim \theclaim. }\label{claim:1}
$|E(T)|\geq |E(J)|+|A|+|B|-3h$.

\smallskip
\noindent\emph{Proof of claim. }
Fix an edge $ef\in E(J)$ such that $e=uv$ and $f=xy$ have no common endpoint in $G$. If neither $e$ selects $f$ nor $f$ selects $e$, then construction \ref{const:3} adds one edge of $S$ having one endpoint in $\{D_u,D_v\}$ and other endpoint in $\{D_x,D_y\}$. If exactly one of $e$ and $f$ selects the other, the corresponding chosen
path contributes two distinct edges of $T$, each having one endpoint in $\{D_u,D_v\}$ and the other in $\{D_x,D_y\}$. Now suppose that $e$ selects $f$ and $f$ selects $e$. Then $e,f\notin Q$. Both chosen paths consequently have length two, because a chosen path of length three between $D_u$ and $D_v$ through both $D_x$ and $D_y$ would contain the edge $D_x-D_y$, implying $f\in Q$. The two length-two paths together contain at least three distinct edges between
$\{D_u,D_v\}$ and $\{D_x,D_y\}$. Thus, each of the $|E(J)|-h$ edges of $J$ considered above gives
one edge of $T$, and each selection of an
edge of the opposite matching by a chosen path gives one 
edge. The number of selections is
$|A|+|B|-2h-|Q|$. Hence $T$ contains at least
$|E(J)|+|A|+|B|-3h-|Q|$ edges joining a disk associated with an edge of $A$ to a disk associated with an edge of $B$. Adding the $|Q|$ edges included by
construction \ref{const:1} proves the claim.
\hfill$\triangleleft$

\medskip
\refstepcounter{claim}
\noindent\textbf{Claim \theclaim. }\label{claim:2}
$T$ is triangle-free.

\smallskip
\noindent\emph{Proof of claim. }
By the construction of $\mathcal B$, every vertex of $T$ is associated with exactly one edge of $A\cup B$. Colour each vertex according to whether its associated edge belongs to $A$ or to $B$. Every edge included in constructions \ref{const:2} and \ref{const:3} of $T$
joins vertices of different colours. The edges included in construction \ref{const:1} join vertices of the same colour and they form a matching. Suppose $e=uv \in A$ and $f=xy \in B$. Therefore, any triangle in $T$ would contain exactly one edge $D_u-D_v$ arising
from some $e\in Q$ and any one disk from $\{D_x, D_y\}$ of edge $f$ adjacent to both $D_u$ and $D_v$. Since $e\in Q$, the edge $e$ does not select $f$. If $f$ does not select $e$, construction~\ref{const:3} includes only one edge having one endpoint in $\{D_u,D_v\}$ and the other in $\{D_x,D_y\}$. If $f$ selects $e$, then, whether its selected path has length two or three, the two edges of that path having one endpoint in $\{D_u,D_v\}$ and the other in $\{D_x,D_y\}$ are incident with different vertices of $\{D_x,D_y\}$. Hence neither $D_x$ nor $D_y$ is adjacent in $T$ to both $D_u$ and $D_v$. Thus, $T$ contains no triangle.
\hfill$\triangleleft$

\medskip
The graph $T$ has $2(|A|+|B|)-2h$ vertices. Since the $h$ edges form a matching and $|A|+|B|\geq 3$, this number is at least three. The graph $T$ is
simple, planar, and triangle-free (from Claim \ref{claim:2}), so $|E(T)|\leq 2|V(T)|-4=4(|A|+|B|)-4h-4$.
Together with Claim \ref{claim:1}, this gives $|E(J)|+|A|+|B|-3h \leq 4(|A|+|B|)-4h-4$, and therefore
$|E(J)|\leq 3(|A|+|B|)-h-4 \leq 3(|A|+|B|)-4$.

Finally, again add the degree-one vertices which were deleted at the beginning. Each added vertex increases $|E(J)|$ by one and increases $|A|+|B|$ by one; hence the
inequality remains valid after every addition of such vertices. This completes the proof.
\hfill $\square$
\end{proof}

We next use Lemma~\ref{lem:ind match} to bound $\delta^*(H)$, where $H=L(G)^2$. For an arbitrary induced subgraph $H[F]$, we partition $F$ according to a fixed $\chi(H)$-colouring of $H$ and apply the lemma to
each pair of nonempty colour classes. 

\begin{theorem}\label{thm:disk-degeneracy} 
For a disk graph $G$ with $\chi'_s(G)\ge2$, 
$\delta^*(L(G)^2)\le6(\chi'_s(G)-1)-1$.
\end{theorem}

\begin{proof}
Let $H=L(G)^2$, and let $\chi(H)$ be the chromatic number\footnote{The chromatic number $\chi(H)$ is the minimum number of colours required to colour the vertices of $H$ so that adjacent vertices receive distinct colours}  of $H$ and  let $\chi(H)=\chi'_s(G)= k$. 
Let $F$ be a non-empty subset of the vertex set $V(H) = E(G)$ and consider the induced subgraph $H[F]$. Since $\chi(H)=k$, assign to each vertex of $H$ one of $k$ colours so that no two adjacent vertices receive the same colour. Partition $F$ into its non-empty colour classes $F_1,\dots,F_r$, where $r\leq k$, according
to this assignment. Each $F_i$ is contained in a colour class of $H$ and hence is an induced matching in $G$.
If $r=1$, then $H[F]$ is edgeless.
For $r\geq2$, let $J_{ij}=(F_i,F_j,E_{ij})$ be the bipartite graph, where $ E_{ij}=\{xy\in E(H):x\in F_i,\ y\in F_j\}$ for each $1 \leq i < j \leq r$. Since each $F_i$ is an independent set in $H$, the sets $E_{ij}$, over all $1\leq i<j\leq r$, form a partition of $E(H[F])$. Applying Lemma~\ref{lem:ind match} to the induced matchings $F_i$ and $F_j$, we obtain
$
|E(H[F])|
=\sum_{1\leq i<j\leq r}|E(J_{ij})|
\leq\sum_{1\leq i<j\leq r}
\bigl(3(|F_i|+|F_j|)-4\bigr)
=3(r-1)|F|-4\binom{r}{2}<3(r-1)|F|.$
Here, $\sum_{1\le i<j\le r}(|F_i|+|F_j|)
=(r-1)\sum_{i=1}^{r}|F_i|=(r-1)|F|$,
because each $|F_i|$ occurs once for each of the other $r-1$ colour classes. Hence the average degree\footnote{The average degree of a nonempty finite graph is the sum of the degrees of its vertices divided by the number of vertices.} of $H[F]$ is strictly less
than $6(r-1)\leq6(k-1)$. Therefore, $H[F]$ contains a vertex of degree
at most $6(k-1)-1$. Since $F$ was arbitrary, every nonempty induced subgraph of $H$
contains such a vertex. Consequently,
$\delta^*(H)\leq6(k-1)-1$.
\hfill $\square$
\end{proof}
The following result immediately follows from Theorem~\ref{thm:disk-degeneracy} and yields a polynomial-time strong edge
colouring algorithm that does not require a disk representation.

\begin{corollary}\label{cor:disk-algorithm}
There is a polynomial-time algorithm on a disk graph
$G$ using at most $\max\{1,6(\chi'_s(G)-1)\}$ colours in a strong edge colouring. Moreover, the algorithm does not require a disk representation of $G$.
\end{corollary}

\begin{proof}
Construct $H=L(G)^2$, obtain an ordering by repeatedly deleting a vertex of minimum degree, and greedily colour the vertices in reverse order. The greedy colouring procedure assigns distinct colours to adjacent vertices of $H$ and therefore yields a strong edge colouring of $G$. If $\chi'_s(G)=1$, then $H$ is edgeless and one
colour is optimal. If $k=\chi'_s(G)\ge2$, Theorem~\ref{thm:disk-degeneracy} implies that
each vertex has at most $6(k-1)-1$ previously coloured neighbours when it is coloured. Hence, the algorithm uses at most $6(k-1)$ colours.

The algorithm does not require a disk representation. Given an adjacency matrix of $G$, whether two edges conflict in a strong edge colouring can be determined in constant
time. Thus $L(G)^2$ can be constructed in $O(|V(G)|^2+|E(G)|^2)$ time. Once $L(G)^2$ has been
constructed, an ordering and the greedy colouring can be computed in $O(|E(G)|^2)$ time. 
\hfill $\square$
\end{proof}

We can now complete the proof of Theorem~\ref{thm:disk}.

\noindent\textit{Proof of Theorem~\ref{thm:disk}.}
If $\chi'_s(G)=1$, the unique colour gives an optimal solution. Otherwise put
$k=\chi'_s(G)\ge2$. By Corollary~\ref{cor:disk-algorithm}, the algorithm uses
at most $6(k-1)$ colours. Hence the algorithm is a polynomial-time $6$-approximation for Strong Edge
Colouring on disk graphs.
\hfill $\square$

\section{Proof of Theorem 2}\label{sec:proofthm2}
In this section, we prove Theorem~\ref{thm:unit disk} by establishing a stronger bound for the strong list chromatic index $\chi'_{s,\ell}$. For each edge $e=uv$, we examine
the subgraph induced by the vertices adjacent to $u$ or $v$ and establish the clique-partition and edge-count bounds needed in the proof. We then order the edges according to the number of copies of $K_4$ containing them and bound the number of earlier conflicting edges. Applying greedy list colouring
to this ordering yields the claimed bound.
Throughout this section, we identify each vertex of $G$ with the centre of its representing disk. Since all representing disks have the same radius, we may scale the representation so that
$uv\in E(G)$ iff $\|u-v\|\leq1$.
We use the following result of Grelier et al.~\cite{grelier2019approximate} in Lemma~\ref{lem:udg-independence}.
\begin{lemma}[\cite{grelier2019approximate}]\label{lem:udg-independence}
    Let $x_1$ and $x_2$ be two points in $\mathbb{R}^2$ within Euclidean distance 1. Let $Y$ be a collection of points in $\mathbb{R}^2$ pairwise of Euclidean distance greater than $1$ such that either $y$ and $x_1$ or $y$ and $x_2$ are within Euclidean distance $1$ for any $y \in Y$. Then $|Y| \leq 8$. 
\end{lemma}
Let $e=uv$ be an edge of a unit disk graph $G$ with maximum degree $\Delta$.
Set $X_1=N(u)\setminus N[v]$, $X_2=N(v)\setminus N[u]$, $X_3=N(u)\cap N(v)$, and $X=X_1\cup X_2\cup X_3$. Recall that, for a graph $G$ we denote its independence number by $\alpha(G)$. We now establish the structural properties of $G[X]$.

\begin{lemma}\label{lem:udg-local}
Each of $G[X_1]$ and $G[X_2]$ can be partitioned into four cliques. Let the vertex sets of the cliques be $C_1,\ldots,C_8$  such that
$|C_1|\leq\cdots\leq |C_8|$.  Then
$\alpha(G[X])\leq8$ and
$|E(G[X])|\geq \frac{|X|^2}{16}- \frac{|X|}{2}$.
Furthermore, if $E_c$ denotes the set of edges joining distinct members of $\{C_1,\ldots,C_8,X_3\}$, then $|E_c|\geq |C_1|\cdot\min\{|C_2|,|X_3|\}$.
\end{lemma}
\begin{proof}
Consider a representation of $G$ by intersecting unit disks with centre $u=(0, 0)$ and $v=(\rho, 0)$, where $0< \rho \leq 1$. A point $x=(r, \theta)$ with $0\leq r \leq 1$ and $|\theta| \leq \frac{\pi}{3}$ satisfies $\| x - v\|^2 \leq r^2 + \rho^2 - r \rho \leq 1$. Hence every point representing a vertex of $G[X_1]$ lies in the complementary angular interval of width at most $\frac{4 \pi}{3}$. Splitting this interval into four sectors of width at most $\frac{\pi}{3}$ gives four cliques. Similar argument also follow for $G[X_2]$.

Every point representing a vertex of $X$ is within distance $1$ of $u$ or
$v$. Lemma~\ref{lem:udg-independence} therefore gives $\alpha(G[X])\leq8$. Applying Tur\'{a}n's theorem to complement of $G[X]$
gives $ |E(G[X])|\geq\binom{|X|}{2}-\frac{7|X|^2}{16}=\frac{|X|^2}{16}-\frac{|X|}{2}.$

If one of $C_1,\ldots,C_8,X_3$ is empty, then $|E_c| \geq 0$. Assume that all nine sets are nonempty. There are $|X_3|\cdot\prod_{i=1}^8|C_i|$ choices of nine vertices, one from each set.
Since $\alpha(G[X])\leq8$, every such choice contains an edge of $E_c$. Counting each edge of $E_c$ once for every choice containing both of its endpoints, the total number of occurrences of edges of $E_c$ among all such choices is at least $|X_3|\cdot\prod_{i=1}^8|C_i|$.
An edge joining two of the nine sets is contained in exactly the total number of choices $|X_3|\cdot \prod_{i=1}^8|C_i|$ divided by the product of the sizes of those two sets. That product is at least $|C_1|\min\{|C_2|,|X_3|\}$. Consequently, each edge of $E_c$ occurs in at most
$\frac{|X_3|\cdot \prod_{i=1}^8|C_i|}{(|C_1|\min\{|C_2|,|X_3|\})}$ choices. Hence, $|X_3|\cdot\prod_{i=1}^8|C_i|\leq\frac{|E_c|\cdot|X_3|\cdot\prod_{i=1}^8|C_i|}
{|C_1|\min\{|C_2|,|X_3|\}} $ and we get that $|E_c|\geq |C_1|\min\{|C_2|,|X_3|\}$.
\hfill $\square$
\end{proof}

For an edge $e=uv$, let
$\kappa(e)=|E(G[N(u)\cap N(v)])|$, which is the number of copies of $K_4$ containing $e$. Order the edges of $G$ by nondecreasing $\kappa$. For the edge $e$, let $b$ be an integer satisfying
$\binom{b-1}{2}\leq\kappa(e)<\binom{b}{2}$. Let $\eta(e)$ denote the number of edges preceding $e$ in this ordering that conflict with $e$.
\begin{lemma}\label{lem:udg-order}
We have
$$
\eta(e)\leq\Delta|X|-\max\Biggl\{
\frac{|X|^2}{16}-\frac{|X|}{2},
\sum_{i=1}^8\binom{|C_i|}{2}
+\sum_{\{i:|C_i|>b\}}\binom{|C_i|}{2}
+\kappa(e)+|C_1|\min\{|C_2|,|X_3|\}
\Biggr\}.
$$
\end{lemma}
\begin{proof}
Suppose that $xy$ is an internal edge of $C_i$. Either the vertex $u$ or the vertex $v$ together with the vertices of $C_i\setminus\{x,y\}$, forms a clique of size $|C_i|-1$ in $N(x)\cap N(y)$. Hence
$\kappa(xy)\geq\binom{|C_i|-1}{2}$. If $|C_i|>b$, then
$\kappa(xy)\geq\binom{b}{2}>\kappa(e)$, so every edge internal to $C_i$ occurs after $e$.
Every edge other than $e$ that conflicts with $e$ has an endpoint in $X$. The number of edges incident with $X$ is at most $\Delta|X|-|E(G[X])|$ and Lemma \ref{lem:udg-local} gives $ |E(G[X])|\geq \frac{|X|^2}{16}-\frac{|X|}{2}$. Again by the above argument we get that,
$\eta(e)\leq\Delta|X|-|E(G[X])|
-\sum_{\{i:|C_i|>b\}}\binom{|C_i|}{2}$.
 Alternatively,
$|E(G[X])|\geq\sum_{i=1}^8\binom{|C_i|}{2}
+|E(G[X_3])|+|E_c|$,
where $|E(G[X_3])|=\kappa(e)$. Now from above two lower bounds of $|E(G[X])|$ we get the required bound of $\eta(e)$.
\hfill $\square$
\end{proof}
The following lemma gives a lower bound for
the second term inside the maximum in
Lemma~\ref{lem:udg-order}, in terms of $\sum_{i=1}^{8}|C_i|$ and $b$.

\begin{lemma}\label{lem:udg quardatic-bound}
Let $0\leq |C_1|\leq\cdots\leq |C_8|$ and let $b\geq0$. Then
$$
\begin{aligned}
&\frac12\sum_{i=1}^8|C_i|^2
+\frac12\sum_{\{i:|C_i|>b\}}|C_i|^2
+|C_1|\min\{|C_2|,b\}\\
&\quad\geq
\begin{cases}
\displaystyle\frac{(\sum_{i=1}^8|C_i|)^2}{14},
&\displaystyle0\leq\sum_{i=1}^8|C_i| \leq 7b,\\[5pt]
\displaystyle3b^2+\frac12(\sum_{i=1}^8|C_i|-6b)^2,
&\displaystyle7b<\sum_{i=1}^8|C_i|\leq8b,\\[5pt]
\displaystyle\frac92b^2+(\sum_{i=1}^8|C_i|-7b)^2, &\displaystyle8b<\sum_{i=1}^8|C_i|< 9b,\\[5pt]
\displaystyle\frac{17}{162}(\sum_{i=1}^8|C_i|)^2,
&\displaystyle\sum_{i=1}^8|C_i|\geq 9b.
\end{cases}
\end{aligned}$$
\end{lemma}
\begin{proof}
Let us consider $\Phi= \frac12\sum_{i=1}^8|C_i|^2
+\frac12\sum_{\{i:|C_i|>b\}}|C_i|^2
+|C_1|\min\{|C_2|,b\} $. If $b=0$, $\Phi$ is at least
$\sum_{i=1}^8|C_i|^2$, and the Cauchy--Schwarz inequality gives $\sum_{i=1}^8|C_i|^2\geq(\sum_{i=1}^8|C_i|)^2/8$. Now let us assume $b>0$.

Suppose $\sum_{i=1}^8|C_i|\leq8b$, and fix the value of this sum. The values $|C_1|,\ldots,|C_8|$ are at most $b$ at a minimizer of $\Phi$. 
Then $\frac12\sum_{\{i:|C_i|>b\}}|C_i|^2 = 0$ and the last term of $\Phi$ is $|C_1||C_2|$. The first two
values $C_1$ and $C_2$ contribute $\frac{(|C_1|+|C_2|)^2}{2}$. By Cauchy--Schwarz inequality, the remaining six
values contribute at least
$\frac{(\sum_{i=1}^8|C_i|-|C_1|-|C_2|)^2}{12}$. Moreover,
$|C_1|+|C_2|\geq\sum_{i=1}^8|C_i|-6b$. Hence $\Phi$ is at least $\frac{(|C_1|+|C_2|)^2}{2}+ \frac{(\sum_{i=1}^8|C_i|-|C_1|-|C_2|)^2}{12}$ and minimized at
$|C_1|+|C_2|=(\sum_{i=1}^8|C_i|)/7$ when $\sum_{i=1}^8|C_i|\leq7b$. Hence  $\Phi \geq \frac{(\sum_{i=1}^8|C_i|)^2}{14}$. When
$7b\leq\sum_{i=1}^8|C_i|\leq8b$ minimum of  $\frac{(|C_1|+|C_2|)^2}{2}+ \frac{(\sum_{i=1}^8|C_i|-|C_1|-|C_2|)^2}{12}$  occurs at
$|C_1|+|C_2|=\sum_{i=1}^8|C_i|-6b$ which gives $\Phi \geq 3b^2+\frac12(\sum_{i=1}^8|C_i|-6b)^2$. 

Now suppose that $\sum_{i=1}^8|C_i|>8b$ and let $t\in\{1,\ldots,7\}$ of the values $|C_1|,\ldots,|C_8|$ exceed $b$. Then the values of $(8-t)$ remaining $|C_i|$'s equal to $b$ at a minimizer of $\Phi$. Since $t$ values have total $|C_i| - (8-t)b$ and the sum of squares is smallest when the values are equal we get that $\Phi$ is at least
$\frac{10-t}{2}b^2+\frac{(\sum_{i=1}^8|C_i|-(8-t)b)^2}{t}$. If $8b<\sum_{i=1}^8|C_i|<9b$, the difference between the expression $\frac{10-t}{2}b^2+\frac{(\sum_{i=1}^8|C_i|-(8-t)b)^2}{t}$ and
its value for $t=1$ is
$\frac{(t-1)(tb^2-2(\sum_{i=1}^8|C_i|-8b)^2)}{2t}\geq0$.
Then $\Phi \geq \frac92b^2+(\sum_{i=1}^8|C_i|-7b)^2 $. Now if all eight values exceed $b$, $\Phi$ is at least
$\frac{(\sum_{i=1}^8|C_i|)^2}{8}$. Since   $8b<\sum_{i=1}^8|C_i|<9b$ we get that, $\frac{(\sum_{i=1}^8|C_i|)^2}{8} > \frac92b^2+(\sum_{i=1}^8|C_i|-7b)^2 $ .

Finally, suppose that $\sum_{i=1}^8|C_i|\geq9b$ and for  $t\in\{1,\ldots,7\}$  subtract $\frac{17}{162}(\sum_{i=1}^8 |C_i|)^2$ from $\frac{10-t}{2}b^2+\frac{(\sum_{i=1}^8|C_i|-(8-t)b)^2}{t}$. Since $\frac{10-t}{2}b^2+\frac{(\sum_{i=1}^8|C_i|-(8-t)b)^2}{t}- \frac{17}{162}(\sum_{i=1}^8 |C_i|)^2 $ is nonnegative at $\sum_{i=1}^8|C_i|\geq9b$ increasing at $\sum_{i=1}^8|C_i|\geq9b$, we get that $\frac{10-t}{2}b^2+\frac{(\sum_{i=1}^8|C_i|-(8-t)b)^2}{t}\geq \frac{17}{162}(\sum_{i=1}^8 |C_i|)^2 $. Hence $\Phi \geq \frac{17}{162}(\sum_{i=1}^8 |C_i|)^2 $. If all the values of $|C_i|$ exceed $b$ then by Cauchy--Schwarz inequality we get that  $ \sum_{i=1}^8 |C_i|^2 \geq \frac{(\sum_{i=1}^8|C_i|)^2}{8} > \frac{17}{162}(\sum_{i=1}^8 |C_i|)^2$.
\hfill $\square$
\end{proof}

We now combine Lemmas~\ref{lem:udg-order} and
\ref{lem:udg quardatic-bound} with the relations below to obtain the inequality used in the final optimisation.
Recall that $|X|=\sum_{i=1}^8|C_i|+|X_3|$. Since $|X_1|+|X_3| \leq \Delta-1$ and $|X_2|+|X_3| \leq \Delta-1$ we get that, $\sum_{i=1}^8|C_i|\leq2(\Delta-|X_3|-1)$. 
If $X_3=\emptyset$, then $\kappa(e)=0$ and hence $b=2$. From Lemma~\ref{lem:udg-order}, $\sum_{i=1}^8\binom{|C_i|}{2}
+\sum_{\{i:|C_i|>2\}}\binom{|C_i|}{2}
\geq \frac{(\sum_{i=1}^8|C_i|)^2}{8}-O(\Delta)$, where we use $\sum_{i=1}^8|C_i|^2\geq \frac{(\sum_{i=1}^8|C_i|)^2}{8}$ and $\sum_{i=1}^8|C_i|\leq2\Delta$. Consequently, $\eta(e)\leq\Delta \sum_{i=1}^8|C_i| -\frac{(\sum_{i=1}^8|C_i|)^2}{8}+O(\Delta)
\leq\frac32\Delta^2+O(\Delta)
<\frac{225}{142}\Delta^2+O(\Delta)$.

Hence, we may assume that $X_3\neq\emptyset$.
Since $X_3\neq\emptyset$, we have $ \binom{b-1} {2}\leq\kappa(e)\leq\binom{|X_3|}{2}$,
and hence $b\leq|X_3|+1$.
Moreover, $\kappa(e)=\frac{b^2}{2}+O(\Delta)$. Expanding the binomial
coefficients in Lemma~\ref{lem:udg-order} and using
$|C_1|\min\{|C_2|,|X_3|\}\geq |C_1|\min\{|C_2|,b\}-O(\Delta)$ gives
\begin{equation}\label{eq:udg-main}
\eta(e)\leq\Delta|X|-\max\Biggl\{
\frac{|X|^2}{16},
\frac12\sum_{i=1}^8|C_i|^2
+\frac12\sum_{\{i:|C_i|>b\}}|C_i|^2
+|C_1|\min\{|C_2|,b\}+\frac{b^2}{2}
\Biggr\}+O(\Delta).
\end{equation}
Optimising the right-hand side of inequality~\ref{eq:udg-main}  subject to the above constraints yields the following bound.

\begin{lemma}\label{lem:udg-optimisation}
For every edge $e$, $\eta(e) \leq \frac{225}{142} \Delta^2 + O(\Delta)$.
\end{lemma}

\begin{proof}
We apply Lemma~\ref{lem:udg quardatic-bound} and consider its four ranges.

\smallskip
\noindent\emph{Case 1: $\sum_{i=1}^8|C_i|\leq7b$.}
Since $b\leq|X_3|+1$, we have $|X|\leq 8|X_3|+7$, and also we have $|X|\leq2\Delta-|X_3|-2$. Consequently, $|X|\leq16\Delta/9 + O(1)$. Since
$\Delta|X|-|X|^2/16$ is increasing on this interval,
$\Delta|X|-\frac{|X|^2}{16}+ O(\Delta)
\leq\frac{128}{81}\Delta^2 +  O(\Delta)
<\frac{225}{142}\Delta^2 +  O(\Delta)$.

\smallskip
\noindent\emph{Case 2: $7b<\sum_{i=1}^8|C_i|\leq8b$.}
Using the second term inside the maximum in inequality~ \ref{eq:udg-main} along with Lemma \ref{lem:udg quardatic-bound} we have, $ \Delta|X|-\frac72b^2-\frac12(\sum_{i=1}^8|C_i|-6b)^2 +  O(\Delta)
\leq \Delta^2+\frac{\Delta}{2}\sum_{i=1}^8|C_i|-\frac72b^2
-\frac12(\sum_{i=1}^8|C_i|-6b)^2 +  O(\Delta)$. As $7b\leq\sum_{i=1}^8|C_i|\leq2(\Delta-b)$ we have
$b\leq2\Delta/9$. Then $\Delta^2+\frac{\Delta}{2}\sum_{i=1}^8|C_i|-\frac72b^2
-\frac12(\sum_{i=1}^8|C_i|-6b)^2 +  O(\Delta)$ is increasing when  $7b<\sum_{i=1}^8|C_i|\leq8b$ taking $b$ as constant. Now since $\sum_{i=1}^8|C_i|\leq8b$ and $\sum_{i=1}^8|C_i|<2 \Delta - 2b$ we get two ranges of $b$. If $b\leq \frac{\Delta}{5}$, substituting $\sum_{i=1}^8|C_i|=8b$ gives
$ \eta(e) \leq \Delta^2+4\Delta b-\frac{11}{2}b^2 +  O(\Delta)
\leq\frac{79}{50}\Delta^2 +  O(\Delta)
<\frac{225}{142}\Delta^2 +  O(\Delta).
$
Here $\Delta^2+4\Delta b-11b^2/2$ is increasing for
$0\leq b\leq\Delta/5$.
Again if $b\geq \frac{\Delta}{5}$, then $2(\Delta-b)\leq8b$, and substituting
$\sum_{i=1}^8|C_i|=2(\Delta-b)$ gives
$ \eta(e) \leq
15\Delta b-\frac{71}{2}b^2 +  O(\Delta)
=\frac{225}{142}\Delta^2
-\frac{71}{2}\left(b-\frac{15\Delta}{71}\right)^2 +  O(\Delta)
\leq\frac{225}{142}\Delta^2 +  O(\Delta).
$

\smallskip
\noindent\emph{Case 3: $8b<\sum_{i=1}^8|C_i|<9b$.}
Since $8b<\sum_{i=1}^8|C_i|< 2(\Delta - b)$ we have that $b < \frac{\Delta}{5}$. Using Lemma \ref{lem:udg quardatic-bound} we get that
$ \eta(e) \leq 
\Delta|X|-5b^2-(\sum_{i=1}^8|C_i|-7b)^2 +  O(\Delta)
\leq\Delta^2+\frac72\Delta b-5b^2
+\frac{\Delta}{2}(\sum_{i=1}^8|C_i|-7b)
-(\sum_{i=1}^8|C_i|-7b)^2 +  O(\Delta)
\leq\Delta^2+\frac72\Delta b-5b^2+\frac{\Delta^2}{16} +  O(\Delta)
\leq\frac{25}{16}\Delta^2 +  O(\Delta)
<\frac{225}{142}\Delta^2 +  O(\Delta).
$
The third inequality follows from
$\frac{\Delta}{2}(\sum_{i=1}^8|C_i|-7b)
-(\sum_{i=1}^8|C_i|-7b)^2\leq\frac{\Delta^2}{16},
$ and the fourth inequality uses $b<\Delta/5$, which gives
$7\Delta b/2-5b^2\leq\Delta^2/2$.

\smallskip
\noindent\emph{Case 4: $\sum_{i=1}^8|C_i|\geq9b$.}
Using Lemma \ref{lem:udg quardatic-bound} and removing the nonnegative term $b^2/2$ gives
$ \eta (e) \leq \Delta|X|-\frac{17}{162}(\sum_{i=1}^8|C_i|)^2 +  O(\Delta)
\leq\Delta^2+\frac{\Delta}{2}\sum_{i=1}^8|C_i|
-\frac{17}{162}(\sum_{i=1}^8|C_i|)^2 +  O(\Delta).$
The expression on the right is increasing for
$0\leq\sum_{i=1}^8|C_i|\leq2\Delta$, so it is at most $128\Delta^2/81 +  O(\Delta)<225\Delta^2/142 +  O(\Delta)$.
\hfill $\square$
\end{proof}
Lemma~\ref{lem:udg-optimisation} bounds, for every edge, the number of preceding edges in
the $\kappa$-ordering that conflict with it. Greedy list colouring in this order now gives the following result.

\begin{theorem}\label{thm:udg-list}
Every unit disk graph $G$ of maximum degree $\Delta$ satisfies
$$ \chi'_{s,\ell}(G)\leq\frac{225}{142}\Delta^2+O(\Delta).$$
Moreover, a corresponding strong list edge colouring of the graph can be found in
polynomial time even if its disk representation is not given.
\end{theorem}

\begin{proof}

Order the edges of $G$ by nondecreasing values of $\kappa$. Consider an arbitrary assignment of list  such that every edge
has a list of at least $1+\max_{e\in E(G)}\eta(e)$ colours. Process the edges in this order and assign to each edge a colour from its list that is not used by any previously coloured edge that conflicts with it. When an edge $e$ is processed, exactly $\eta(e)$ preceding edges conflict with $e$. These edges forbid at most $\eta(e)$ colours from the list of $e$. Since the list contains at least $\eta(e)+1$ colours, an available colour for edge $e$ always exists. The resulting colouring is therefore a strong list edge colouring of $G$. Consequently, $ \chi'_{s,\ell}(G)\leq 1+\max_{e\in E(G)}\eta(e).$ By Lemma \ref{lem:udg-optimisation}, $ \eta(e)\leq\frac{225}{142}\Delta^2+O(\Delta)$. Therefore, $ \chi'_{s,\ell}(G)\leq\frac{225}{142}\Delta^2+O(\Delta).$

For every edge $e=uv$, the value
$\kappa(e)=|E(G[N(u)\cap N(v)])|$ can be computed directly from any graph $G$. Hence, the edges of $G$ can be sorted by their $\kappa$-values, and the previously coloured edges conflicting with each edge can be checked in polynomial time. Thus, a strong list edge colouring can be found in
polynomial time without a unit disk representation.
\hfill $\square$
\end{proof}

\noindent\textit{Proof of Theorem~\ref{thm:unit disk}. }
Since $\chi'_s(G)\leq\chi'_{s,\ell}(G)$ for every graph $G$,
Theorem~\ref{thm:udg-list} gives $
\chi'_s(G)\leq\frac{225}{142}\Delta^2+O(\Delta).
$ \hfill $\square$

\section{Conclusion}\label{sec:conclusion}

We studied Strong Edge Colouring on disk graphs and unit disk graphs. For disk graphs, we obtained a polynomial-time $6$-approximation algorithm that does not require a geometric representation of the input graph. For unit disk graphs, we improved the best known asymptotic upper bound on the strong chromatic index from $13\Delta^2/8$ to
$\frac{225}{142}\Delta^2+O(\Delta)$, and in fact established the same bound for the strong list chromatic index.

Several natural questions remain open. For arbitrary disk graphs, it would be interesting to improve the approximation factor below $6$, thereby reducing the gap with the known inapproximability threshold. For unit disk graphs, there is still a considerable gap between the known quadratic lower bound with leading coefficient $9/16$ and our upper coefficient $225/142$. Closing this gap, either by improving the upper bound or by constructing stronger lower-bound examples, is an interesting direction for future work.

\bibliographystyle{splncs04}
\bibliography{ref}

\end{document}